\documentclass[11pt]{amsart}
\usepackage[margin=1.35in]{geometry}
\usepackage{amscd,amsmath,amsxtra,amsthm,amssymb,stmaryrd,xr,mathrsfs,mathtools,enumerate,commath, comment}
\usepackage{stmaryrd}
\usepackage{xcolor}
\usepackage{commath}
\usepackage{comment}
\usepackage{tikz-cd}
\usepackage{longtable} 
\usepackage{pdflscape} 
\usepackage{booktabs}
\usepackage{hyperref}
\definecolor{vegasgold}{rgb}{0.77, 0.7, 0.35}
\definecolor{darkgoldenrod}{rgb}{0.72, 0.53, 0.04}
\definecolor{gold(metallic)}{rgb}{0.83, 0.69, 0.22}
\hypersetup{
 colorlinks=true,
 linkcolor=darkgoldenrod,
 filecolor=brown,      
 urlcolor=gold(metallic),
 citecolor=darkgoldenrod,
 pdftitle={Class group statistics for the Torsion fields generated by elliptic curves},
 }

\usepackage[all,cmtip]{xy}

\DeclareFontFamily{U}{wncy}{}
\DeclareFontShape{U}{wncy}{m}{n}{<->wncyr10}{}
\DeclareSymbolFont{mcy}{U}{wncy}{m}{n}
\DeclareMathSymbol{\Sh}{\mathord}{mcy}{"58}
\usepackage[T2A,T1]{fontenc}
\usepackage[OT2,T1]{fontenc}

\newtheorem{theorem}{Theorem}[section]
\newtheorem{lemma}[theorem]{Lemma}

\newtheorem{conjecture}[theorem]{Conjecture}
\newtheorem*{theorem*}{Theorem}
\newtheorem*{corollary*}{Corollary}

\newtheorem{proposition}[theorem]{Proposition}
\newtheorem{corollary}[theorem]{Corollary}
\newtheorem{definition}[theorem]{Definition}

\numberwithin{equation}{section}

\theoremstyle{remark}
\newtheorem{remark}[theorem]{Remark}

\newcommand{\I}{\op{I}}

\newcommand{\Z}{\mathbb{Z}}

\newcommand{\Q}{\mathbb{Q}}
\newcommand{\F}{\mathbb{F}}

\newcommand{\Sel}{\mathrm{Sel}}

\newcommand{\W}{\mathcal{W}}

\newcommand{\op}[1]{\operatorname{#1}}

\newcommand\mtx[4] { \left( {\begin{array}{cc}
 #1 & #2 \\
 #3 & #4 \\
 \end{array} } \right)}

\begin{document}
\title[Class group Statistics for torsion fields]{Class group statistics for torsion fields generated by elliptic curves}

\author[A.~Ray]{Anwesh Ray}
\address[Ray]{Chennai Mathematical Institute, H1, SIPCOT IT Park, Kelambakkam, Siruseri, Tamil Nadu 603103, India}
\email{ar2222@cornell.edu}

\author[T.~Weston]{Tom Weston}
\address[T.~Weston]{Department of Mathematics, University of Massachusetts, Amherst, MA, USA.} 
  \email{weston@math.umass.edu}

\begin{abstract}
For a prime $p$ and a rational elliptic curve $E_{/\Q}$, set $K=\Q(E[p])$ to denote the torsion field generated by $E[p]:=\op{ker}\{E\xrightarrow{p} E\}$. The class group $\op{Cl}_K$ is a module over $\op{Gal}(K/\Q)$. Given a fixed odd prime number $p$, we study the average non-vanishing of certain Galois stable quotients of the mod-$p$ class group $\op{Cl}_K/p\op{Cl}_K$. Here, $E$ varies over all rational elliptic curves, ordered according to \emph{height}. Our results are conditional, since we assume that the $p$-primary part of the Tate-Shafarevich group is finite. Furthermore, we assume predictions made by Delaunay for the statistical variation of the $p$-primary parts of Tate-Shafarevich groups. We also prove results in the case when the elliptic curve $E_{/\Q}$ is fixed and the prime $p$ is allowed to vary. 
\end{abstract}


\maketitle
 
\section{Introduction}\label{s 1}
Given a family of number fields, it is natural to study the statistical variation of class groups in the family. Of significant interest is the family of imaginary quadratic extensions $K/\Q$. Such investigations go back to the work of Gauss, who was interested in the determination of all imaginary quadratic fields of a given class number $h$. The problem was solved for $h=1$ by Baker, Heegner and Stark, for $h=2$ by Baker and Stark and for $h=3$ by Oesterl\'e. Watkins in \cite{watkins2004class} computed the imaginary quadratic fields for which $h\leq 100$. Soundararajan showed that the number of imaginary quadratic fields of class number $<x$ is asymptotically $\frac{3\zeta(2)}{\zeta(3)}x^2$ as $x\rightarrow \infty$. The story for real quadratic fields is significantly different. Predictions can be made for the distribution of class groups via Cohen-Lenstra heuristics for all quadratic number fields \cite{cohen1984heuristics} and Cohen-Lenstra-Martinet heuristics for general number fields \cite{martinet1990etude}. We refer to \cite{wood2016asymptotics} for a survey of these methods. The study of class numbers on average in families of number fields has notably gained significant momentum in the following works \cite{davenport1971density, wong1999rank, bhargava2005density,fouvry20074, ellenberg2007reflection,ellenberg2017}. On the other hand, there has been much interest in the study of the $p$-ranks of class groups of the cyclotomic extension $\Q(\mu_p)$. The Herbrand-Ribet theorem (cf. \cite{washington}) establishes a precise relationship between generalized Bernoulli numbers and the $p$-rank of the class group of $\Q(\mu_p)$. Ribet's converse relates congruences between Eisenstein series and cusp forms to the existence of $p$-cyclic unramified extensions of $\Q(\mu_p)$. The rank of the Eisenstein ideal is closely related to the Galois module structure of mod-$p$ class groups that arise from certain families of number fields of the form $\Q(N^{\frac{1}{p}})$, where $N\equiv 1\mod{p}$ is a prime, cf. \cite{MazurEisensteinIdeal, Lecouturier, SchStubleyclassgroups, WakeWangErickson}. It is however difficult to resolve distribution questions in this context, though there have been some computational experiments done (cf. \cite[Introduction]{WakeWangErickson}).

\par Motivated by such developments, we consider the family of number field extensions that are generated by the torsion in elliptic curves defined over $\Q$. To be specific, given an elliptic curve $E_{/\Q}$, and a prime $p$, we let $E[p]$ denote the $p$-torsion subgroup of $E(\bar{\Q})$. Let $K=\Q(E[p])$ be the field generated by $E[p]$, defined to be the fixed field of the residual representation 
\[\bar{\rho}:\op{Gal}(\bar{\Q}/\Q)\rightarrow \op{GL}_2(\F_p)\]on $E[p]\simeq \F_p\oplus \F_p$. The extension $K/\Q$ in general is a nonabelian extension that contains $\Q(\mu_p)$. Let $\op{Cl}_K$ be the class group of $K$. Note that $G=\op{Gal}(K/\Q)$ acts on $\op{Cl}_K$ and thus on the mod-$p$ class group $\op{Cl}_K/p\op{Cl}_K$. We study Galois stable quotients of the $\op{Cl}_K/p\op{Cl}_K$ that are isomorphic to $E[p]$ as a $G$-module. In some sense, such investigations generalize the study of the $p$-rank of the class group of $\Q(\mu_p)$. To be precise, we study the following related questions.
\begin{enumerate}
    \item Let $p$ be a fixed odd prime. As $E$ varies over all elliptic curves defined over $\Q$, how often is $\op{Hom}_{G}\left(\op{Cl}_K/p\op{Cl}_K, E[p]\right)$ non-vanishing?
     \item Let $E_{/\Q}$ be a fixed elliptic curve. As $p$ varies over all primes, how does the dimension of $\op{Hom}_{G}\left(\op{Cl}_K/p\op{Cl}_K, E[p]\right)$ vary? 
\end{enumerate}
In the context of the first question above, we order elliptic curves $E_{/\Q}$ according to \emph{height}. We shall assume throughout that the $p$-primary part of the Tate-Shafarevich group is finite. Our results also rely on heuristics for the statistical behavior of Tate-Shafarevich groups of elliptic curves. These heuristics were studied by Delaunay \cite{delaunay2007heuristics}. Let $\mathcal{E}_p$ be the set of elliptic curves $E_{/\Q}$ such that $\op{Hom}_G(\op{Cl}_K, E[p])\neq 0$. Given $\Delta\in \Z_{<0}$ with $\Delta\equiv 0,1\mod{4}$, let $H(\Delta)$ be Hurwitz class number, see Definition \ref{def hurwitz class number}. Given a set of elliptic curves $\mathcal{S}$, we denote by $\underline{\mathfrak{d}}(\mathcal{S})$ the lower density of Weierstrass equations for elliptic curves in $\mathcal{S}$ (see \eqref{def of density} for the definition of $\underline{\mathfrak{d}}(\mathcal{S})$).
\begin{theorem*}[Theorem \ref{main thm}]
Let $p$ be an odd prime, and assume that Conjecture \ref{conjecture delaunay} is satisfied. Then, we have that 
\[\underline{\mathfrak{d}}(\mathcal{E}_p)>(p^{-1}+p^{-3}-p^{-4})(1-p^{-1}-\mathfrak{d}_p-\mathfrak{d}_p'),\] where
\[\mathfrak{d}_p=\begin{cases}
\zeta(p)-1 & \text{ if }p\geq 5,\\
(\zeta(3)-1)+(\zeta(4)-1)+(\zeta(7)-1)& \text{ if }p=3,\\
\end{cases}\]
\[\mathfrak{d}_p'=\begin{cases}\left( \frac{p-1}{2p^2}\right)H(1-4p) & \text{ if }p\geq 7,\\
\left(\frac{p-1}{2p^2}\right)\left(H(1-4p)+H(p^2+1-6p)\right) & \text{ if }p\leq 5.\\
\end{cases}\]
\end{theorem*}

We obtain the following Corollary to the above result.
\begin{corollary*}[Corollary \ref{main cor new}]
Let $p$ be an odd prime, and assume that Conjecture \ref{conjecture delaunay} is satisfied. Then, we have that\[\underline{\mathfrak{d}}(\mathcal{E}_p)\geq  p^{-1}+O(p^{-3/2+\epsilon}).\] In other words, for any choice of $\epsilon>0$ there is a constant $C>0$, depending on $\epsilon$ and independent of $p$, such that \[\underline{\mathfrak{d}}(\mathcal{E}_p)\geq p^{-1}-Cp^{-3/2+\epsilon}.\]
\end{corollary*}
\par Next, we study the second question, where $E$ is fixed and $p$ varies. In this context, we prove two results.

\begin{theorem*}[Theorem \ref{th 5.1}]
Let $E_{/\Q}$ be an elliptic curve and assume that the following conditions are satisfied, 
\begin{enumerate}
    \item $\Sh(E/\Q)$ is finite, 
    \item $\op{rank} E(\Q)\leq 1$, 
    \item $E$ does not have complex multiplication. 
\end{enumerate}
Then, for $100\%$ of primes $p$, we have that $\op{Hom}_G(\op{Cl}_K/p\op{Cl}_K, E[p])=0$.
\end{theorem*}

\begin{theorem*}[Theorem \ref{th 5.2}]

Let $E_{/\Q}$ be an elliptic curve and assume that the following conditions are satisfied, 
\begin{enumerate}
    \item $\Sh(E/\Q)$ is finite,
    \item $\op{rank} E(\Q)\geq 2$, 
    \item $E$ does not have complex multiplication. 
\end{enumerate}
Then, for $100\%$ (i.e., for a Dirichlet density $1$ set) of primes $p$, we have that \[\op{dim}\op{Hom}_G(\op{Cl}_K/p\op{Cl}_K, E[p])\geq \op{rank} E(\Q)-1.\]
\end{theorem*}
\par \emph{Organization:} In section \ref{s 2}, we introduce preliminary notions and also recall results of Prasad and Shekhar \cite{prasad2021relating} on the Galois module structure of the class group of the torsion field $\Q(E[p])$. In section \ref{s 3}, we study the density of elliptic curves satisfying local conditions at a possibly infinite set of primes. Such results are crucially used in proving the main results of the article. In section \ref{s 4}, we fix an odd prime $p$. The main results are Theorem \ref{main thm} and Corollary \ref{main cor new}, which establish that $\op{Hom}_G(\op{Cl}_K/p\op{Cl}_K, E[p])$ is non-vanishing for a positive density set of elliptic curves. Furthermore, we prove a lower bound for this density explicitly. These results are conditional since it relies on the heuristic of Delaunay. In section \ref{s 5}, we fix an elliptic curve $E_{/\Q}$ which does not have complex multiplication and study the variation of the dimension of $\op{Hom}_G\left(\op{Cl}_K/p\op{Cl}_K, E[p]\right)$, where $p$ varies over all prime numbers. Finally, in section \ref{s 6} we provide explicit computations for the prime $p=3$. These computations do also illustrate cases of interest, when some of the hypotheses on $E$ and $p$ are relaxed.

\subsection*{Acknowledgment} From September 2022 to September 2023, the first author's research is supported by the CRM Simons postdoctoral fellowship.

\section{Preliminaries}\label{s 2}
\par Fix an algebraic closure $\bar{\Q}$ of $\Q$ and let $E$ be an elliptic curve defined over $\Q$. Let $p$ be a prime number and set $\F_p=\Z/p\Z$ to denote the finite field with $p$ elements. For $n\geq 1$, denote by $E[p^n]$ the kernel of the multiplication map $\times p^n: E(\bar{\Q})\rightarrow E(\bar{\Q})$, and set $E[p^\infty]:=\bigcup_{n} E[p^n]$. The absolute Galois group $\op{G}_{\Q}:=\op{Gal}(\bar{\Q}/\Q)$ acts on $E[p^n]$. Choosing an isomorphism $E[p^n]\simeq \left(\Z/p^n \Z\right)^2$, let $\rho_n$ be the associated Galois representation on $E[p^n]$, \[\rho_n:\op{G}_{\Q}\rightarrow \op{GL}_2(\Z/p^n\Z).\] Set $\rho:=\varprojlim_n \rho_n$ to be the Galois representation on the $p$-adic Tate-module $\op{T}_p(E):=\varprojlim_n E[p^n]$. Set $\bar{\rho}$ to denote the mod-$p$ reduction of the characteristic-zero representation $\rho:\op{G}_{\Q}\rightarrow \op{GL}_2(\Z_p)$. We shall refer to $\bar{\rho}$ as the \emph{residual representation} and identify $\rho_1:\op{G}_{\Q}\rightarrow \op{GL}_2(\F_p)$ with $\bar{\rho}$. Let $\chi:\op{G}_{\Q}\rightarrow \Z_p^\times$ be the $p$-adic cyclotomic character, and $\bar{\chi}:\op{G}_{\Q}\rightarrow \F_p^\times$ its mod-$p$ reduction.
\par Given an integer $r$, set $\Z_p(r):=\Z_p(\chi^r)$ and given a $\Z_p[\op{G}_{\Q}]$-module $M$, set $M(r):=M\otimes_{\Z_p} \Z_p(\chi^r)$ to denote the \emph{$r$-th Tate-twist} of $M$. Note that if $M$ is an $\F_p[\op{G}_{\Q}]$-module, then, $M(r)=M\otimes_{\F_p} \F_p(r)$. Let $V_{\bar{\rho}}=E[p]$ be the underlying vector space for $\bar{\rho}$. Given $i,j\in \Z$ lying in the range $1\leq i\leq (p-1)$ and $1\leq j\leq (p-2)$, we set $V^{i,j}=V^{i,j}_{p,E}$ to denote $\op{Sym}^i(V_{\bar{\rho}})(j)$. Throughout, we shall assume that $\bar{\rho}$ is irreducible. However, for the sake of discussion, let us consider the special case when $\bar{\rho}$ is surjective. Then, upon choosing a basis of $V_{\bar{\rho}}=E[p]$, we may identify the image of $\bar{\rho}$ with $\op{GL}_2(\F_p)$ and thus the module $V^{i,j}$ is viewed as an irreducible representation of $\op{GL}_2(\F_p)$. Note that any irreducible representation of $\op{GL}_2(\F_p)$ over $\F_p$ is of the form $V^{i,j}$. A semisimple module $M$ over $\F_p[\op{G}_{\Q}]$ decomposes into a direct sum 
\[M\simeq \bigoplus_{i,j} \left(V^{i,j}\right)^{r_{i,j}(M)},\] where $r_{i,j}(M)\geq 0$.

\par Let $K=\Q(E[p])$ be the \emph{splitting field} of $\bar{\rho}$, taken to be the Galois extension of $\Q$ given by $K=\bar{\Q}^{\op{ker}\bar{\rho}}$. Set $G:=\op{Gal}(K/\Q)$, and identify $G$ with the image of $\bar{\rho}$. Note that when $\bar{\rho}$ is surjective, $G$ is identified with $\op{GL}_2(\F_p)$. There is a natural action of $G$ on the class group $\op{Cl}_K$, and thus on the mod-$p$ class group $\op{Cl}_K/p\op{Cl}_K$. The study of the Galois module structure of $\op{Cl}_K/p\op{Cl}_K$ shall be the primary focus of this paper. When $\bar{\rho}$ is surjective, its semi-simplification decomposes into a direct sum of irreducible representations of $\op{GL}_2(\F_p)$
\[\left(\op{Cl}_K/p \op{Cl}_K\right)^{\op{ss}}\simeq \bigoplus_{i,j} \left(V^{i,j}\right)^{n_{i,j}}.\] We are specifically interested in the space of homomorphisms $\op{Hom}_G\left(\op{Cl}_K/p\op{Cl}_K, E[p]\right)$. Note that $\op{Hom}_G\left(\op{Cl}_K/p\op{Cl}_K, E[p]\right)\neq 0$ precisely when there is a $G$-stable quotient of $\op{Cl}_K/p \op{Cl}_K$ which is isomorphic to $E[p]$. By class field theory, this corresponds to the existence of an unramified $(\Z/p\Z)^2$-extension $L$ of $K$ which is Galois over $\Q$, such that $\op{Gal}(L/K)\simeq E[p]$ as a module over $G=\op{Gal}(K/\Q)$. We have that
\[n_{1,0}\geq \op{dim}_{\F_p} \op{Hom}_G\left(\op{Cl}_K/p\op{Cl}_K, E[p]\right),\] with equality in the special case when the representation of $G$ on $\op{Cl}_K/p\op{Cl}_K$ is semisimple. 
\par Given a prime number $\ell$, let $\mathcal{E}$ be the N\'eron model of $E$ over $\Z_\ell$. Let $E^0(\Q_\ell)$ be the subset of points of $E(\Q_\ell)=\mathcal{E}(\Z_\ell)$ which reduce modulo $\ell$ to the identity component of $E_{/\F_\ell}$. Fix an absolute value $|\cdot |_p:\Q_p^\times\rightarrow \Q^\times$, normalized by $|p|_p^{-1}=p$. The \emph{Tamagawa number} at $\ell$ is set to be $c_\ell(E):=[E(\Q_\ell):E^0(\Q_\ell)]$, and set $c_\ell^{(p)}(E):=|c_\ell(E)|_p^{-1}$. Therefore, $c_\ell(E)$ is a unit in $\Z_p$ if and only if $c_\ell^{(p)}(E)=1$. We denote by $\op{Sel}_p(E/\Q)$ the $p$-Selmer group of $E$ over $\Q$ defined by
\[\op{Sel}_p(E/\Q):=\op{ker}\left\{H^1(\bar{\Q}/\Q, E[p])\rightarrow \prod_{\ell} H^1\left(\bar{\Q}_\ell/\Q_\ell, E(\bar{\Q}_\ell)\right)[p]\right\},\] where the restriction map for the prime $\ell$ is the composite
\[H^1(\bar{\Q}/\Q, E[p])\rightarrow H^1(\bar{\Q}_\ell/\Q_\ell, E[p])\rightarrow H^1(\bar{\Q}_\ell/\Q_\ell, E(\bar{\Q}_\ell))[p].\]
The Tate-Shafarevich group 
\[\Sh(E/\Q):=\op{ker}\left\{H^1(\bar{\Q}/\Q, E(\bar{\Q}))\rightarrow \prod_{l} H^1(\bar{\Q}_l/\Q_l, E(\bar{\Q}_l))\right\},\]
fits into an exact sequence 
\[0\rightarrow E(\Q)/p E(\Q)\rightarrow \op{Sel}_p(E/\Q)\rightarrow \Sh(E/\Q)[p]\rightarrow 0,\] see \cite[p.8, (1)]{coates2000galois}. \par Note that $\op{dim}_{\F_p} E(\Q)/p E(\Q)=\op{rank} E(\Q)+\op{dim}_{\F_p} E(\Q)[p]$, and since $\bar{\rho}$ is irreducible, $E(\Q)[p]=0$, and thus, $\op{dim}_{\F_p} E(\Q)/p E(\Q)=\op{rank} E(\Q)$. Hence, 
we find that 
\[\op{dim}_{\F_p} \op{Sel}_p(E/\Q)=\op{rank} E(\Q)+\dim_{\F_p} \Sh(E/\Q)[p].\]
Due to the Cassels-Tate pairing, the $\F_p$-dimension of $\Sh(E/\Q)[p]$ is even. The following result shows that there is an explicit relationship between $\op{Cl}_K/p\op{Cl}_K$ and the $p$-torsion in the Tate-Shafarevich group $\Sh(E/\Q)$ and
is key to constructing quotients of $\op{Cl}_K/p\op{Cl}_K$ that are isomorphic to $E[p]$. 

\begin{theorem}[Prasad-Shekhar]\label{psthm1}
Let $E$ be an elliptic curve over $\Q$ and $p$ be an odd prime number at which $E$ has good reduction. With respect to notation above, suppose that the following assumptions hold:
\begin{enumerate}
    \item $c_\ell^{(p)}(E)=1$ for all primes $\ell\neq p$,
    \item $E(\Q_p)[p]=0$.
\end{enumerate}
Then, we have that 
\begin{equation}\label{main eqn}\op{dim}_{\F_p}\op{Hom}_G\left(\op{Cl}_K/p\op{Cl}_K, E[p]\right)\geq \op{rank} E(\Q)+\dim_{\F_p} \Sh(E/\Q)[p]-1.\end{equation}
In particular, under the above assumptions, if $\Sh(E/\Q)[p]\neq 0$, or $\op{rank} E(\Q)\geq 2$, then, $\op{Hom}\left(\op{Cl}_K/p\op{Cl}_K, E[p]\right)$ is non-zero.
\end{theorem}
\begin{proof}
The inequality \eqref{main eqn} follows from \cite[Theorem 3.1]{prasad2021relating}. The second assertion follows from Corollary 3.2 of \emph{loc. cit.}
\end{proof}

\begin{definition}\label{Ip defn}
Let $\mathcal{I}=\mathcal{I}_p(E)$ be the set of primes $\ell\neq p$ such that 
\begin{itemize}
    \item $E$ has split multiplicative reduction at $\ell$ and $E(\Q_\ell)[p]$ has rank $1$, 
    \item $E$ has non-split multiplicative reduction at $\ell$ and $E(\Q_\ell)[p]$ and $E(\Q_\ell^{\op{nr}})[p]$ has rank $1$.
\end{itemize}
\end{definition}

\begin{theorem}[Prasad-Shekhar]\label{psth2}
Let $E_{/\Q}$ be an elliptic curve and $p$ an odd prime at which the conditions of Theorem \ref{psthm1} are satisfied. Let $\mathcal{I}$ be the finite set from Definition \ref{Ip defn}. Then, the following bound is satisfied
\[\op{dim}_{\F_p}\op{Hom}\left(\op{Cl}_K/p\op{Cl}_K, E[p]\right)\leq \op{rank}E(\Q)+\op{dim}_{\F_p}\Sh(E/\Q)[p]-1 +\#\mathcal{I}.\]
\end{theorem}
\begin{proof}
The above result is \cite[Theorem 4.2]{prasad2021relating}.
\end{proof}

\par Theorem \ref{psthm1} is crucially used in the proof of Theorem \ref{main thm}, which is our main result. On the other hand, both Theorems \ref{psthm1} and \ref{psth2} are applied to prove results in the case when $E$ is a fixed elliptic curve and $p$ varies over all primes at which $E$ has good reduction, see section \ref{s 5}.
\section{Density results for Weierstrass equations with local conditions}\label{s 3}
\par In this section, we recall results due to Cremona and Sadek \cite{cremona2020local} for the density of elliptic curves $E_{/\Q}$ satisfying local conditions at a prescribed set of primes. This set of primes may in fact be infinite. The results in this section are used to study the proportion of elliptic curves $E_{/\Q}$ ordered according to \emph{height} which satisfy the conditions of Theorem \ref{psthm1}. As always, we fix an odd prime $p$. Given a tuple of integers $a=(a_1,a_2,a_3,a_4,a_6)$ we have an associated elliptic curve $E_a$ defined by the long Weierstrass equation 
\[Y^2+a_1XY+a_3Y=X^3+a_2X^2+a_4X+a_6.\]

The \emph{height} of $E_a$ is defined as follows
\[\op{ht}(a)=\op{ht}(E_a):=\op{max}_i\{|a_i|^{\frac{1}{i}}\}.\] 

We let 
\[\begin{split}
    & b_2=a_1^2+4 a_2, b_4=a_1 a_3+2 a_4, b_6=a_3^2+4a_6,\\
    & b_8=a_1^2 a_6+4a_2a_6-a_1a_3a_4+a_2 a_3^2-a_4^2, \\
    &\Delta(a)=-b_2^2 b_8-8 b_4^3-27 b_6^2+9b_2 b_4 b_6.
\end{split}\]
Given a ring $R$, we let 
\[\W(R) = R^5 = \{a=(a_1,a_2,a_3,a_4,a_6)|a_i\in R\},\]
and let $\Delta(a)=\Delta(E_a)$ be the associated discriminant. Given a prime $\ell$, let $\W_{M}(\Z_\ell)$ be the subset of $\W(\Z_\ell)$ consisting of equations which are \emph{minimal} in the sense of \cite[section \rm{VII}.1]{silverman2009arithmetic}. Recall that if $E$ is an elliptic curve over $\Q_\ell$, the \emph{Kodaira type} is one of the following choices
\begin{itemize}
  \item $\op{I}_0$ if $E$ has good reduction;
  \item $\op{I}_{\ge1}$ if $E$ has bad multiplicative reduction, of type $\I_m$ for some
  $m\ge1$.
  \item Finally, if $E$ has bad additive reduction, we have the following choices:
     $\rm{II}$, $\rm{III}$, $\rm{IV}$, $\rm{II}^*$, $\rm{III}^*$, $\rm{IV}^*$, $\rm{I}_0^*$,
    $\rm{I}_{\ge1}^*$, the latter meaning type $\rm{I}_m^*$ for some $m\geq 1$.
\end{itemize} We refer to \cite[\rm{IV}, section 9]{silverman1994advanced} for a comprehensive study of Kodaira types and how they may be detected via Tate's algorithm. For $a\in \W(\Z_\ell)$, we say that $a$ has Kodaira type $T$ if the associated elliptic curve $E_a$ has Kodaira type $T$. Note that this definition applies to non-minimal Weierstrass equations as well. Given a Kodaira type $T$, let $\W^T(\Z_\ell)$ be the subset of $\W(\Z_\ell)$ of tuples $a$ of Kodaira type $T$, and set $\W_{M}^T(\Z_\ell)=\W^T(\Z_\ell)\cap \W_{M}(\Z_\ell)$. Let $\mu$ be the Haar measure on $\W(\Z_\ell)\simeq \Z_\ell^5$, and set $\rho^M(\ell):=\mu\left(\W_{M}(\Z_\ell)\right)$. Given a Kodaira type $T$ at the prime $\ell$, set \[\begin{split}
& \rho^M_T(\ell):=\mu\left(\W_M^T(\Z_\ell)\right),\\
& \rho_T(\ell):=\mu\left(\W^T(\Z_\ell)\right).\\
\end{split}\]According to \cite[Proposition 2.6]{cremona2020local}, we have that $\rho_T(\ell)=(1-\ell^{-10})^{-1}\rho^M_T(\ell)$. These local densities are calculated in \emph{loc. cit.}, and these calculations are summarized here.
\begin{theorem}[Cremona-Sadek]\label{th cremona sadek}
Let $\ell$ be a prime, then $\rho^M(\ell)=1-\ell^{-10}$. For each Kodaira type $T$, the local density $\rho_T^M(\ell)$ is given by the following values:
\[
\begin{tabular}{clc}
  \hline
  $T$ & $\rho_T^M(\ell)$ \\
  \hline
  $\rm{I}_{0}$ & $(\ell-1)/\ell$ \\
  $\rm{II}$ &  $(\ell-1)/\ell^3$ \\

  $\rm{III}$ &  $(\ell-1)/\ell^4$ \\

  $\rm{IV}$ &  $(\ell-1)/\ell^5$ \\

  $\rm{I}_0^*$ & $(\ell-1)/\ell^6$ \\

  $\rm{I}_{\ge1}^*$ & $(\ell-1)/\ell^7$ \\

  $\rm{IV}^*$ & $(\ell-1)/\ell^8$ \\

  $\rm{III}^*$ & $(\ell-1)/\ell^9$ \\

  $\rm{II}^*$ & $(\ell-1)/\ell^{10}$ \\
   $\rm{I}_m$ & $(\ell-1)^2/\ell^{m+2}$ \\
    $\rm{I}_{\geq m}$ & $(\ell-1)/\ell^{m+1}$. \\
  \hline
\end{tabular}
\]

\end{theorem}

\begin{proof}
The above result follows from \cite[Propositions 2.1 (3), 2.2, 2.5]{cremona2020local}.
\end{proof}
\par Let $\mathcal{S}$ be a subset of $\W(\Z)\simeq \Z^5$, the \emph{density} of $\mathcal{S}$ is taken to be the following limit
\begin{equation}\label{def of density}\begin{split}\mathfrak{d}(\mathcal{S})
&:=\lim_{x\rightarrow \infty} \frac{\#\{a\in \mathcal{S}\mid \op{ht}(a)<x\}}{\#\{a\in \Z^5\mid \op{ht}(a)<x\}},\\
&=\lim_{x\rightarrow \infty} 2^{-5}x^{-16}\#\{a\in \mathcal{S}\mid |a_i|<x^i \text{ for }i=1,2,3,4,6\}.\\\end{split}\end{equation} Note that the above limit may not exist. We let $\overline{\mathfrak{d}}(\mathcal{S})$ and $\underline{\mathfrak{d}}(\mathcal{S})$ be the upper and lower limit respectively. Note that $\overline{\mathfrak{d}}(\mathcal{S})$ and $\underline{\mathfrak{d}}(\mathcal{S})$ are defined by replacing the limit by $\limsup$ and $\liminf$ respectively, and that these limits do exist unconditionally. If $\mathcal{S}$ is a set of elliptic curves $E_{/\Q}$, we shall by abuse of notation denote by $\mathcal{S}$ the subset of $\W(\Z)$ consisting of all tuples $a=(a_1,a_2,a_3,a_4,a_6)$ such that $E_a\in \mathcal{S}$. Note that since we work with long Weierstrass equations, the choice of $a$ for a given isomorphism class of elliptic curves is not unique. According to \cite[Theorem 1.1]{cremona2020local} the proportion of integral Weierstrass equation that are globally minimal is $1/\zeta(10)=93555/\pi^{10}\approx 99.99\%$. Let $\Phi$ be a possibly infinite set of prime numbers and for each prime $\ell\in \Phi$, let $U_\ell$ be a subset of $\W(\Z_\ell)$ defined by a set of congruence classes. In other words, $U_\ell$ is the inverse image of a subset of $\W(\Z/\ell^M \Z)$ for some integer $M>0$. Let $U$ be the family of conditions $\{U_\ell\mid \ell\in \Phi\}$. Given $N>0$, let \[Z_N(U):=\{a\in \W(\Z)\mid a\in U_\ell \text{ for some prime }\ell>N\}.\]

\begin{definition}
The family $U=\{U_\ell\mid \ell\in \Phi\}$ is said to be \emph{admissible} if \[\lim_{N\rightarrow \infty} \mathfrak{d}\left(Z_N(U)\right)=0.\]
\end{definition}

Let $\W_U$ be the set of integral Weierstrass equations $a\in \W(\Z)$ not satisfying $U_\ell$ at any prime $\ell\in \Phi$.

\begin{proposition}\label{formula for density}
Let $U$ be an admissible family and $s_\ell=\mu(U_\ell)$. Then, $\sum_{\ell\in \Phi} s_\ell$ converges and 
\[\mathfrak{d}(\W_U)=\prod_{\ell\in \Phi} (1-s_\ell).\]
\end{proposition}
\begin{proof}
The above result follows from \cite[Proposition 3.4]{cremona2020local}.
\end{proof}

\begin{definition}
Given a Kodaira type $T$, let $U(T)$ be the family of conditions on the set of primes $\ell\neq p$, where $U_\ell$ consists of local Weierstrass equations $a\in \W(\Z_\ell)$ satisfying $T$. Thus, $\W_{U(T)}$ consists of integral Weierstrass equations $a\in \W(\Z)$ such that $E_a$ does not satisfy $T$ at any prime $\ell\neq p$.
\end{definition}

\begin{lemma}\label{density wut}
Let $T$ be a Kodaira type such that $U(T)$ is admissible, then, the density of $\W_{U(T)}$ exists and equals
\[\mathfrak{d}(\W_{U(T)})=\prod_{\ell\neq p}\left(1-\rho_T(\ell)\right)\geq 1-\sum_{\ell\neq p} \rho_T(\ell).\]
\end{lemma}

\begin{proof}
The result follows directly from Proposition \ref{formula for density}.
\end{proof}
\begin{theorem}\label{thm kod types}
Let $p$ and $\ell$ be distinct prime numbers, and $E$ an elliptic curve over $\Q_\ell$. The following assertions hold.
\begin{enumerate}
    \item\label{c1 thm kod types} Suppose that $p\geq 5$, then, $p|c_{\ell}(E)$ if and only if $E$ has Kodaira type $T=\op{I}_{pm}$ for $m\in \Z_{\geq 1}$.
    \item\label{c2 thm kod types} Suppose that $p=3$, then, $p|c_{\ell}(E)$ if and only if $E$ has Kodaira type $T=\op{I}_{3m}$ for $m\in \Z_{\geq 1}$, or $T=\op{IV}$, or $T=\op{IV}^*$.
\end{enumerate}
\end{theorem}

\begin{proof}
The result is well known, see \cite[Table 4.1, p.365]{silverman1994advanced}.
\end{proof}
We come to the main result of this section, which is subsequently used in the proof of our main result in the next section.
\begin{definition}\label{def of sp}
 Recall that $p$ is a fixed prime number. Let $\mathcal{S}_{p}$ be the set of elliptic curves $E_{/\Q}$ such that $p|c_\ell(E)$ for some prime $\ell\neq p$.
\end{definition}
\begin{theorem}\label{technical thm}
 Let $E$ be an elliptic curve and $p$ an odd prime number. We have the following assertions
\begin{enumerate}
    \item if $p\geq 5$, then, \[\overline{\mathfrak{d}}(\mathcal{S}_{p})< \zeta(p)-1\]
    \item if $p=3$, then, \[\begin{split}\overline{\mathfrak{d}}(\mathcal{S}_{p}) & <(\zeta(3)-1)+(\zeta(4)-1)+(\zeta(7)-1).\end{split}\]
\end{enumerate}
\end{theorem}
\begin{proof}
We prove the result on a case by case basis.
\begin{enumerate}
    \item First, we consider the case when $p\geq 5$. Then, by Theorem \ref{thm kod types} \eqref{c1 thm kod types}, we have that $p|c_{\ell}(E)$ if and only if the Kodaira type $T$ is $\op{I}_{mp}$ for some integer $m\geq 1$. Let $U=U(\op{I}_{\geq p})$ be the datum such that $U_\ell=I_{\geq p}$ at every prime $\ell\neq p$. It follows directly from the proof of \cite[Theorem 4.6]{cremona2020local} that the datum $U(\op{I}_{\geq 2})$ is admissible, and hence, so is $U$. From Lemma \ref{density wut}, we have that
    \[\begin{split}\mathfrak{d}(\W_U)=&\prod_{\ell\neq p} \left(1-(1-\ell^{-10})^{-1}(1-\ell^{-1})\ell^{-p}\right)\\
    \geq & \prod_{\ell\neq p}(1-\ell^{-p})\\
    \geq & 1-\sum_{\ell\neq p} \ell^{-p}.\end{split}\]
    Since $\mathcal{S}_p$ is contained in the complement of $\W_U$, we find that 
    \[\overline{\mathfrak{d}}(\mathcal{S}_p)\leq \sum_{\ell\neq p} \ell^{-p}<\zeta(p)-1.\]
    \item We now consider the case when $p=3$. By Theorem \ref{thm kod types} \eqref{c2 thm kod types}, we have that $3|c_{\ell}(E)$ if and only if the Kodaira type $T$ is $\op{I}_{3m}$ for some integer $m\geq 1$, or, $T=\rm{IV}$ or $T=\rm{IV}^*$. Let $U$ be such that $U_\ell$ is the subset of $\W(\Z_\ell)$ with Kodaira type $\op{I}_{\geq 3}$, $\rm{IV}$ or $\rm{IV}^*$ respectively. Note that $\ell^2$ divides $\Delta$ for such reduction types. Admissibility is a direct consequence of the proof of \cite[Theorem 4.6]{cremona2020local}.
    
    We find that 
    \[\begin{split}\mathfrak{d}(\W_U)\geq 
    &\prod_{\ell\neq p} \left(1-\rho_{\op{I}_{\geq 3}}(\ell)-\rho_{\rm{IV}}(\ell)-\rho_{\rm{IV}^*}(\ell)\right)\\
    \geq 
    &\prod_{\ell\neq p} \left(1-(1-\ell^{-10})^{-1}(1-\ell^{-1})(\ell^{-3}+\ell^{-4}+\ell^{-7})\right)\\
    \geq 
    &\prod_{\ell\neq p} \left(1-(\ell^{-3}+\ell^{-4}+\ell^{-7})\right)\\
    \geq & 1-\sum_{\ell\neq p} \ell^{-3}-\sum_{\ell\neq p} \ell^{-4}-\sum_{\ell\neq p} \ell^{-7}.\\
    > & 1-(\zeta(3)-1)-(\zeta(4)-1)-(\zeta(7)-1).
    \end{split}\]
Thus, we find that 
 \[\overline{\mathfrak{d}}(\mathcal{S}_p)<(\zeta(3)-1)+(\zeta(4)-1)+(\zeta(7)-1).\]
\end{enumerate}
\end{proof}
\begin{definition}\label{def spprime}Denote by $\mathcal{S}_p'$ the set of elliptic curves $E$ for which $p\nmid \Delta(E)$ and $E(\Q_p)[p]\neq 0$. Note that for such elliptic curves and $p>2$, $\widetilde{E}(\F_p)[p]\neq 0$, where $\widetilde{E}$ is the reduction of $E$.
\end{definition}

Let $\Delta\in \Z_{<0}$ with $\Delta\equiv 0,1\mod{4}$, and set
\[B(\Delta):=\left\{ax^2+bxy+cy^2\in \Z[x,y]\mid  a>0, b^2-4ac=\Delta\right\}.\] The group $\op{SL}_2(\Z)$ acts on $\Z[x,y]$, where a matrix $\sigma=\mtx{p}{q}{r}{s}$ sends $x\mapsto (px+qy)$ and $y\mapsto (rx+sy)$. Thus, if $f=ax^2+bxy+cy^2$, the matrix $\sigma$ acts by
\[f\circ \sigma= a(px+qy)^2+b(px+qy)(rx+sy)+c(rx+sy)^2.\] We note that $B(\Delta)$ is stable under the action of $\op{SL}_2(\Z)$ and $B(\Delta)/\op{SL}_2(\Z)$ is finite, see \cite{schoof1987nonsingular} for additional details.
\begin{definition}\label{def hurwitz class number}
The \emph{Hurwitz class number} $H(\Delta)$ is the order of $B(\Delta)/\op{SL}_2(\Z)$.
\end{definition}

Let $d(p)$ be the number of tuples $a\in \W(\F_p)$ such that $E_a(\F_p)[p]\neq 0$. Note that \[\frac{d(p)}{\#\W(\F_p)}=\frac{d(p)}{p^5}.\]

\begin{lemma}
With respect to notation above, $\overline{\mathfrak{d}}(\mathcal{S}_p')\leq \frac{d(p)}{p^5}$.
\end{lemma}
\begin{proof}
Given $a\in \W(\F_p)$, let $\mathcal{S}_{a}$ be the set of $\tilde{a}\in \W(\Z)$ which reduce to $a$. It is easy to see that $\mathfrak{d}(\mathcal{S}_{a})= \frac{1}{p^5}$. Note that $\mathcal{S}_p'$ is contained in the disjoint union $\bigsqcup_{a}\mathcal{S}_{a}$, where $a$ ranges over the tuples in $\W(\F_p)$ such that $E_{a}(\F_p)[p]\neq 0$. The result follows from this.
\end{proof}

\begin{proposition}\label{main prop}
With respect to notation above, we have that
\[\overline{\mathfrak{d}}(\mathcal{S}_p')\leq \begin{cases}\left( \frac{p-1}{2p^2}\right)H(1-4p) & \text{ if }p\geq 7,\\
\left(\frac{p-1}{2p^2}\right)\left(H(1-4p)+H(p^2+1-6p)\right) & \text{ if }p\leq 5.\\
\end{cases}\]
\end{proposition}
\begin{proof}
The number of isomorphism classes of elliptic curves $E$ over $\F_p$ such that $E(\F_p)[p]\neq 0$ is \[\begin{cases}H(1-4p) & \text{ if }p\geq 7,\\
H(1-4p)+H(p^2+1-6p) & \text{ if }p\leq 5,\\
\end{cases}\]
see \cite[Corollary 3.11]{ray2021arithmetic} for additional details.
Given $a\in \W(\F_p)$ and associated integral Weierstrass model $E_a$, then, after a transformation $(X,Y)\mapsto (X+r, Y+sX+t)$, we obtain an elliptic curve with a short Weierstrass equation. Note that two elliptic curves 
\[\begin{split} & E_{c,d}: Y^2=X^3+cX+d\\
& E_{c',d'}: Y^2=X^3+c'X+d'\end{split}\]are isomorphic over $\F_p$ if there exists $s\in \F_p^\times$ such that \[c'=s^4 c\text{ and }d'=s^6 d.\]It is thus easy to see that the number of Weierstrass equations in a given isomorphism class is at most  $p^3\left(\frac{p-1}{2}\right)$. Thus, we find that 
\[d(p)\leq \begin{cases}p^3\left(\frac{p-1}{2}\right)H(1-4p) & \text{ if }p\geq 7,\\
p^3\left(\frac{p-1}{2}\right)\left(H(1-4p)+H(p^2+1-6p) \right) & \text{ if }p\leq 5.\\
\end{cases}\]
\end{proof}
\section{Results for a fixed prime and varying elliptic curve}\label{s 4}
\par In this section we prove our main results. Throughout, we fix an odd prime number $p$. Given a set of rational elliptic curves $\mathcal{S}$, we obtain a subset of $\W(\Z)$ consisting of all $a$ such that $E_a\in \mathcal{S}$. By abuse of notation, we shall refer to this set as $\mathcal{S}$ as well. The density $\mathfrak{d}(\mathcal{S})$ is defined as in \eqref{def of density}. Recall that the upper and lower densities, denoted $\overline{\mathfrak{d}}(\mathcal{S})$ and $\underline{\mathfrak{d}}(\mathcal{S})$, are always defined even if $\mathfrak{d}(\mathcal{S})$ need not be.  Let $\mathcal{E}_p$ be the set of elliptic curves $E_{/\Q}$ such that $\op{Hom}_G(\op{Cl}_K, E[p])\neq 0$. Note that for $E\in \mathcal{E}_p$, we have in particular that $p|\#\op{Cl}_K$.
\par Recall from Definition \ref{def of sp} that $\mathcal{S}_{p}$ is the set of elliptic curves $E_{/\Q}$ such that $p|c_\ell(E)$ for some prime $\ell\neq p$. Recall from Definition \ref{def spprime} that $\mathcal{S}_p'$ is the set of elliptic curves $E$ for which $p\nmid \Delta(E)$ and $E(\Q_p)[p]\neq 0$. Denote by $\mathcal{S}_p''$ the set of elliptic curves $E_{/\Q}$ with bad reduction at $p$. Let $\mathfrak{T}_p$ be the elliptic curves $E$ such that $\Sh(E/\Q)[p^\infty]$ is finite and $\Sh(E/\Q)[p]\neq 0$.

\begin{proposition}\label{prop 4.1}
Let $p$ be an odd prime. Then, we have that \[\underline{\mathfrak{d}}(\mathcal{E}_p)\geq \underline{\mathfrak{d}}\left(\mathfrak{T}_p\backslash (\mathcal{S}_p\cup \mathcal{S}_p'\cup \mathcal{S}_p'')\right).\]
\end{proposition}
\begin{proof}
Suppose $E$ is an elliptic curve in $\mathfrak{T}_p\backslash (\mathcal{S}_p\cup \mathcal{S}_p'\cup \mathcal{S}_p'')$, then, since $E\in \mathfrak{T}_p$, we have that $\Sh(E/\Q)[p]\neq 0$, and since $E\notin \mathcal{S}_p\cup \mathcal{S}_p'\cup \mathcal{S}_p''$, we have that $p\nmid c_\ell(E)$ for all primes $\ell\neq p$, and $E(\Q_p)[p]=0$. So long as  the residual representation on $E[p]$ is irreducible it follows from Theorem \ref{psthm1} that $\op{Hom}_G(\op{Cl}_K, E[p])\neq 0$. As a result, $E$ is contained in $\mathcal{E}_p$. By a well known result of Duke \cite{duke1997elliptic}, the residual representation on $E[p]$ is irreducible for $100\%$ of elliptic curves, and the result follows.
\end{proof}
Throughout, we shall assume that for any elliptic curve $E_{/\Q}$, the $p$-primary part of the Tate-Shafarevich group $\Sh(E/\Q)[p^\infty]$ is finite.
The density of elliptic curves $E_{/\Q}$ such that $\Sh(E/\Q)[p^\infty]\neq 0$ was studied by Delaunay \cite{delaunay2007heuristics}.
\begin{conjecture}[Delaunay]\label{conjecture delaunay}
With respect to notation above, \[\underline{\mathfrak{d}}(\mathfrak{T}_p)\geq 1-\prod_{i\geq 1}\left(1-p^{-(2i-1)}\right)> p^{-1}+p^{-3}-p^{-4}.\] Moreover, given a set $\mathcal{E}$ of elliptic curves defined by local congruence conditions, 
\[\underline{\mathfrak{d}}(\mathfrak{T}_p\cap \mathcal{E})=\underline{\mathfrak{d}}(\mathfrak{T}_p)\mathfrak{d}(\mathcal{E}).\]
\end{conjecture}
\begin{remark}
For the purposes of this article, we shall use only the second lower bound predicted by the above conjecture: that is, we shall assume that 
\[\underline{\mathfrak{d}}(\mathfrak{T}_p\cap \mathcal{E}) >\left( p^{-1}+p^{-3}-p^{-4} \right)\mathfrak{d}(\mathcal{E}),\] where $\mathcal{E}$ is taken to be the complement of $\mathcal{S}_p\cup \mathcal{S}_p'\cup \mathcal{S}_p''$.
\end{remark}
\begin{theorem}\label{main thm}
Let $p$ be an odd prime, and assume the above conjecture. Then, we have that 
\[\underline{\mathfrak{d}}(\mathcal{E}_p)>(p^{-1}+p^{-3}-p^{-4})(1-p^{-1}-\mathfrak{d}_p-\mathfrak{d}_p'),\] where
\[\mathfrak{d}_p=\begin{cases}
\zeta(p)-1 & \text{ if }p\geq 5,\\
(\zeta(3)-1)+(\zeta(4)-1)+(\zeta(7)-1)& \text{ if }p=3,\\
\end{cases}\]
\[\mathfrak{d}_p'=\begin{cases}\left( \frac{p-1}{2p^2}\right)H(1-4p) & \text{ if }p\geq 7,\\
\left(\frac{p-1}{2p^2}\right)\left(H(1-4p)+H(p^2+1-6p)\right) & \text{ if }p\leq 5.\\
\end{cases}\]
\end{theorem}
\begin{proof}
Let $\mathcal{E}$ be the set of elliptic curves in the complement of $\mathcal{S}_p\cup \mathcal{S}_p'\cup \mathcal{S}_p''$. We have from Theorem \ref{technical thm} (resp. Proposition \ref{main prop}) that $\mathfrak{d}(\mathcal{S}_p)=\mathfrak{d}_p$ (resp. $\mathfrak{d}(\mathcal{S}_p')=\mathfrak{d}_p'$). Furthermore, we have from Theorem \ref{th cremona sadek} and Proposition \ref{formula for density} that $\mathfrak{d}(\mathcal{S}_p'')=\frac{1}{p}$. From Proposition \ref{prop 4.1}, we have that $\underline{\mathfrak{d}}(\mathcal{E}_p)\geq \underline{\mathfrak{d}}\left(\mathfrak{T}_p\backslash (\mathcal{S}_p\cup \mathcal{S}_p'\cup \mathcal{S}_p'')\right)$. According to our assumption,
\[\begin{split}
&\underline{\mathfrak{d}}\left(\mathfrak{T}_p\backslash (\mathcal{S}_p\cup \mathcal{S}_p'\cup \mathcal{S}_p'')\right)=\underline{\mathfrak{d}}(\mathfrak{T}_p\cap \mathcal{E})\\
\geq & \left(1-\prod_{i\geq 1}\left(1-p^{-(2i-1)}\right)\right)\mathfrak{d}(\mathcal{E})\\
> &\left( p^{-1}+p^{-3}-p^{-4} \right)\mathfrak{d}(\mathcal{E})\\
=& \left( p^{-1}+p^{-3}-p^{-4} \right) \left(1-p^{-1}-\mathfrak{d}_p-\mathfrak{d}_p'\right),\end{split}\] and the result follows.

\end{proof}
\begin{corollary}\label{main cor new}
Assume Conjecture~\ref{conjecture delaunay}.  Then
 \[\underline{\mathfrak{d}}(\mathcal{E}_p)\geq  p^{-1}+O(p^{-3/2+\epsilon}).\] In other words, for any choice of $\epsilon>0$ there is a constant $C>0$, depending on $\epsilon$ and independent of $p$, such that \[\underline{\mathfrak{d}}(\mathcal{E}_p)\geq p^{-1}-Cp^{-3/2+\epsilon}.\]
\end{corollary}
\begin{proof}
According to Theorem \ref{main thm}, we have that 
\[\begin{split} \underline{\mathfrak{d}}(\mathcal{E}_p)> & (p^{-1}+p^{-3}-p^{-4})(1-p^{-1}-\mathfrak{d}_p-\mathfrak{d}_p')\\
> & p^{-1}-p^{-1}\left(p^{-1}+\mathfrak{d}_p+\mathfrak{d}_p'\right).
\end{split}\]

We have that
 \[\zeta(p)-1=2^{-p}+\sum_{n\geq 3} n^{-p}< 2^{-p}+\int_{2}^\infty x^{-p}dx=2^{-p}\left(\frac{p+1}{p-1}\right),\] and hence, \[\mathfrak{d}_p=O(2^{-p})=O(p^{-1/2+\epsilon}).\]
 On the other hand, it follows from known results that \[\mathfrak{d}_p'=O\left(p^{-1/2}\op{log}p(\op{log} \op{log} p)^2\right)=O(p^{-1/2+\epsilon}),\] (see \cite[Proposition 1.9]{lenstra1987factoring}) and the result follows.
\end{proof}

\section{Results for a fixed elliptic curve and varying prime}\label{s 5}
\par In this section we prove results for a fixed elliptic curve $E_{/\Q}$ and varying prime $p$. We shall assume throughout that $E$ does not have complex multiplication, and that the Tate-Shafarevich group $\Sh(E/\Q)$ is finite. Given a prime $p$, we let $K_p=\Q(E[p])$. We make a number of observations with regard to the variation of $\op{Cl}_{K_p}/p\op{Cl}_{K_p}$, where $p$ varies over all primes $p$.

\begin{theorem}\label{th 5.1}
Let $E_{/\Q}$ be an elliptic curve, assume that the following conditions are satisfied, 
\begin{enumerate}
    \item $\Sh(E/\Q)$ is finite, 
    \item $\op{rank} E(\Q)\leq 1$, 
    \item $E$ does not have complex multiplication. 
\end{enumerate}
Then, for $100\%$ of primes $p$, we have that $\op{Hom}_G(\op{Cl}_{K_p}/p\op{Cl}_{K_p}, E[p])=0$.
\end{theorem}

\begin{proof}
Note that for all but finitely many primes $p$, 
\begin{enumerate}
    \item $E$ has good reduction at $p$,
    \item $E(\Q_\ell)[p]=0$ for all primes $\ell$ at which $E$ has bad reduction,
    \item $\Sh(E/\Q)[p]=0$,
    \item the representation on $E[p]$ is irreducible,
    \item $c_\ell^{(p)}(E)=1$ for all primes $\ell$ at which $E$ has bad reduction.
\end{enumerate}
Thus, it follows from Theorem \ref{psth2}, that we have that $\op{Hom}_G(\op{Cl}_{K_p}/p\op{Cl}_{K_p}, E[p])=0$ provided $E(\Q_p)[p]=0$. Let $\widetilde{E}$ be the reduction of $E$ over $\F_p$. The formal group $\widehat{E}(\Z_p)$ of $E$ at $p$ has no nontrivial $p$-torsion. Thus, if $p$ is a prime of good reduction, the natural map $E(\Q_p)\rightarrow \widetilde{E}(\F_p)$ induces an injection 
\[E(\Q_p)[p]\rightarrow \widetilde{E}(\F_p)[p].\] A prime $p$ at which $E$ has good reduction is called an \emph{anomalous prime} if $\widetilde{E}(\F_p)[p]\neq 0$. It is well known that for a non-CM elliptic curve, $100\%$ of primes are non-anomalous (see for instance \cite{murty1997modular}). Thus, $E(\Q_p)[p]=0$ for $100\%$ of primes $p$, and thus the result follows.
\end{proof}

\begin{theorem}\label{th 5.2}
Let $E_{/\Q}$ be an elliptic curve, assume that the following conditions are satisfied, 
\begin{enumerate}
    \item $\Sh(E/Q)$ is finite,
    \item $\op{rank} E(\Q)\geq 2$, 
    \item $E$ does not have complex multiplication. 
\end{enumerate}
Then, for $100\%$ of primes $p$, we have that \[\op{dim}\op{Hom}_G(\op{Cl}_{K_p}/p\op{Cl}_{K_p}, E[p])\geq \op{rank} E(\Q)-1.\]
\end{theorem}
\begin{proof}
It follows from the proof of Theorem \ref{th 5.1} that for $100\%$ of primes $p$, all of the following conditions are satisfied:
\begin{enumerate}
    \item $E$ has good reduction at $p$,
    \item the representation on $E[p]$ is irreducible,
    \item $c_\ell^{(p)}(E)=1$ for all primes $\ell$ at which $E$ has bad reduction,
    \item $E(\Q_p)[p]=0$.
\end{enumerate}
It follows from Theorem \ref{psthm1} that \[\op{dim}\op{Hom}_G(\op{Cl}_{K_p}/p\op{Cl}_{K_p}, E[p])\geq \op{rank} E(\Q)-1.\]
\end{proof}

\section{Computational results}\label{s 6}

In this section we present some computations of $\op{Cl}_{K}/3\op{Cl}_K$, where $K=\Q(E[3])$, for certain families of elliptic curves $E$. Due to the difficulty of computing statistics for larger extensions, we primarily restricted ourselves to the case 
that $\op{Gal}(K/\Q)\subset \op{GL}_2(\F_3)$
is the normalizer of a split Cartan subgroup. This is the smallest irreducible subgroup of $\op{GL}_2(\F_3)$.  We include also a few additional calculations where the Galois group is the normalizer of a non-split Cartan subgroup.

We note that the hypotheses of \cite{prasad2021relating} are often violated in these cases. Specifically, a majority of such $E$ have bad reduction at $3$. In addition, all ramification at $3$ is tame as the Galois group has order prime to $3$: thus any such $E$ for which $a_3(E) \equiv 1 \pmod{3}$ has $3$ as a local torsion prime. We computed our examples without regard to these considerations.  We comment only that we verified that every computation which violates the lower bound of \cite{prasad2021relating} does not satisfy their hypotheses.

\subsubsection{Normalizer of a split Cartan: Varying $j$-invariants}

We use the formula of Zywina \cite[Theorem 1.2]{zywinapossible}
$$j_2(t) = 27\frac{(t+1)^3(t-3)^3}{t^3}$$
such that any elliptic curve with such a $j$-invariant has $\op{Gal}(\Q(E[3])/\Q)$ contained in the normalizer of a split Cartan subgroup.  For any $t$, let
$E_t$ denote the elliptic curve of smallest conductor with $j$-invariant $j_2(t)$ and let $K_t$ denote the field $\Q(E_t[3])$.  We computed
 $\op{Sel}_3(E_t/\Q)$ as well as 
the Galois module structure of $\op{Cl}_{K_t}/3$
for the 599 elliptic curves in the set
$$S_0 = \{E_{a/b} \mid |a| \leq 50, 1 \leq b \leq 10, [K_t:\Q]=8 \}.$$

Let $L_t$ denote the unique biquadratic subfield of $K_t$; note that $L_t$ necessarily contains
$\Q(\sqrt{-3})$.   In each case there was $d_t \geq 0$ such that there was an isomorphism
$$\op{Cl}_{K_t}/3 \cong \op{Cl}_{L_t}/3 \oplus E_t[3]^{d_t}$$
as $\op{Gal}(K_t/\Q)$-modules.  Note that $\op{Gal}(K_t/\Q)$ acts on $\op{Cl}_{L_t}/3$ via its abelianization.  

Our interest here is primarily in comparing $\op{Sel}_3(E_t/\Q)$ and $d_t$.  We briefly report the data for $\op{Cl}_{L_t}/3$:

\centerline{\begin{tabular}{c|c}
$\dim \op{Cl}_{L_t}/3$ & \# $t$ \\ \hline
0 & 447 \\
1 & 123 \\
2 & 28 \\
3 & 1
\end{tabular}}

\noindent
The only curve in our sample with $\dim \op{Cl}_{L_t}/3 = 3$ was $E_{-46/3}$ with $L_t = \Q(\sqrt{-3},\sqrt{2917})$.

Turning to the other quantities, we have the following.

\centerline{\begin{tabular}{c|c||c}
$\dim \op{Sel}_3(E_t/\Q)$ & $d_t$ & \# $t$ \\ \hline
0 & 0 & 103 \\
0 & 1 & 26 \\
1 & 0 & 228 \\
1 & 1 & 52 \\
1 & 2 & 2 \\
2 & 0 & 95 \\
2 & 1 & 74 \\
2 & 2 & 2 \\
3 & 0 & 4 \\
3 & 1 & 12 \\
3 & 2 & 1
\end{tabular}}

We remark that the restriction to $E$ with specified $3$-torsion fields obviously biases the distribution of $3$-Selmer groups.
Although it is unwise to extrapolate too much from this limited data, we note that for $t$ with $\dim \op{Sel}_3(E_t/\Q) \leq 1$, the dimension of
$\op{Cl}_{K_t}/3$ appears to be roughly the same (being nonzero around $19\%$ of the time) whether the Selmer group has dimension zero or one.  By contrast, for Selmer groups of dimension two, the proportion having $d_t \geq 1$ is much larger at $44\%$.  For Selmer groups of dimension three it is much larger than that, in this limited sample.

Granting the limits of this data set, it certainly appears that $\op{Sel}_3(E_t/\Q)$ has a significant effect on $d_t$, but there are additional influences as well which are not currently clear.

\subsubsection{Normalizer of a split Cartan: Fixed $j$-invariants}

We also investigated two quadratic twist families.  Let $E_1$ denote the elliptic curve
$$y^2+xy+y = x^3-141x+624$$
with $j(E_1) = 857375/8$ and conductor $10082=2 \cdot 71^2$.  For any squarefree $t$, let $E_{1,t}$ denote the quadratic twist of $E_1$ by $t$
and let $K_{1,t}$ denote the field $\Q(E_{1,t}[3])$.  Then
$\op{Gal}(K_{1,t}/\Q)$ is the normalizer of a split Cartan subgroup of $\op{GL}_2(\F_3)$.
We note that each $K_{1,t}$ contains the biquadratic field $L_1=\Q(\sqrt{-3},\sqrt{-71})$ which has trivial $3$-class group.

Consider the set of 608 elliptic curves
$$S_1 = \left\{ E_{1,t} \mid 1 \leq t \leq 1000, t \text{~squarefree} \right\}.$$
In each case we computed a $\op{Gal}(K_{1,t}/\Q)$-isomorphism
$$\op{Cl}_{K_{1,t}}/3 \cong E_{1,t}[3]^{d_{1,t}}$$
for some $d_{1,t}$.
We computed the following data.

\centerline{
\begin{tabular}{c|c||c}
$\dim \Sel_3(E_{1,t}/\Q)$ & $d_{1,t}$ & \# $t$ \\ \hline
0 & 0 & 341 \\
0 & 1 & 1 \\
1 & 0 & 74 \\
1 & 1 & 31 \\
1 & 2 & 2 \\
2 & 1 & 140 \\
2 & 2 & 18 \\
3 & 2 & 1 
\end{tabular}}

We note here that the dependence of $d_{1,t}$ on $\Sel_3(E_{1,t}/\Q)$ is more clear than in the case of varying $j$-invariants: the lower bound of \cite{prasad2021relating} (although as noted the hypotheses are often violated in this data set) appears to control most of the behavior of $d_{1,t}$.

We computed analogous data for the elliptic curve $E_2$ given by
$$y^2 = x^3-83667346875x - 10711930420406250$$
with $j(E_2)= -42875/8$ and
conductor $6962 = 2 \cdot 59^2$.  With notation as above, we consider the set of 608 quadratic twists
$$S_2 = \left\{ E_{2,t} \mid 1 \leq t \leq 1000, t \text{~squarefree} \right\}.$$
Each $3$-division field field $K_{2,t}$ has Galois group the normalizer of a split Cartan subgroup and contains the biquadratic field $L_2=\Q(\sqrt{-3},\sqrt{-59}).$  This field has class group
$\Z/3$, and for each $t$ we computed a $\op{Gal}(K_{2,t}/\Q)$-isomorphism
$$\op{Cl}_{K_{2,t}}/3 \cong \op{Cl}_{L_2}/3 \oplus E_{2,t}[3]^{d_{2,t}}$$
for some $d_{2,t}$.

\centerline{
\begin{tabular}{c|c||c}
$\dim \Sel_3(E_{2,t}/\Q)$ & $d_{2,t}$ & \# $t$ \\ \hline
0 & 0 & 54 \\
0 & 1 & 18 \\
1 & 0 & 332 \\
1 & 1 & 165 \\
2 & 0 & 18 \\
2 & 1 & 13 \\
2 & 2 & 2 \\
3 & 2 & 6
\end{tabular}}

This behavior appears somewhat less straightforward than for $S_1$.  The distribution of $\Sel_3(E_{2,t}/\Q)$ is obviously quite different than that of $\Sel_3(E_{1,t}/\Q)$ and the effect on $d_{2,t}$ is less clear.  We do note that the distributions of $d_{1,t}$ and $d_{2,t}$ ignoring Selmer groups are similar, which is somewhat curious.

\subsubsection{Normalizer of a non-split Cartan: fixed $j$-invariant}

We compiled very limited data in one case with mod $3$ image of Galois the normalizer of a non-split Cartan subgroup.  Consider the elliptic curve
$E_3$ given by
$$y^2 = x^3+x^2-2x-8$$
with $j(E_3)=-64$ and
conductor $1568=2^5 \cdot 7^2$.  We consider its
set of 61 twists
$$S_3 = \left\{ E_{3,t} \mid 1 \leq t \leq 100, t \text{~squarefree}\right\}.$$
For each $t$ we have an isomorphism of $\op{Gal}(K_{3,t}/\Q)$-modules
$$\op{Cl}_{K_{3,t}}/3 \cong E_{3,t}[3]^{d_{3,t}}$$
for some $d_{3,t}$.

\centerline{
\begin{tabular}{c|c||c}
$\dim \Sel_3(E_{3,t}/\Q)$ & $d_{3,t}$ & \# $t$ \\ \hline
0 & 0 & 18 \\
0 & 1 & 2 \\
1 & 0 & 18 \\
1 & 1 & 14 \\
1 & 2 & 1 \\
2 & 1 & 8
\end{tabular}}

There is  too little data here to provide anything more than wild speculations; unfortunately the computations rapidly became very time consuming beyond this point.

\bibliographystyle{alpha}
\bibliography{references}
\end{document}